\documentclass[10pt, reqno]{amsart}
\usepackage[utf8]{inputenc}
\usepackage{bbm,amssymb,mathtools,mathrsfs}
\usepackage[dvipsnames]{xcolor}
\usepackage{graphicx,titletoc}
\usepackage[titletoc]{appendix}
\definecolor{darkblue}{rgb}{0.0, 0.0, 0.55}
\definecolor{bordeaux}{rgb}{0.34, 0.01, 0.1}
\usepackage[hyphens]{url}
\usepackage[colorlinks,breaklinks,linkcolor=bordeaux,citecolor=darkblue,urlcolor=black,hypertexnames=true]{hyperref}
\mathtoolsset{showonlyrefs,showmanualtags}

\usepackage[shortlabels]{enumitem}

\newtheorem{thm}{Theorem}[section]
\newtheorem{cor}[thm]{Corollary}
\newtheorem{lem}[thm]{Lemma}
\newtheorem{prop}[thm]{Proposition}
\newtheorem{conj}[thm]{Conjecture}
\theoremstyle{definition}

\newtheorem{rem}[thm]{Remark}
\newtheorem{exa}[thm]{Example}

\DeclareMathOperator{\rk}{rank}

\DeclareMathOperator{\id}{id}

\def\R{\mathbb R}
\def\Q{\mathbb Q}
\def\Z{\mathbb Z}
\def\N{\mathbb N}

\def\ve{\varepsilon}

\newcommand{\Langle}{\mathop{<}\!}
\newcommand{\Rangle}{\!\mathop{>}}
\newcommand{\cx}[1]{[x_1,\dots,x_{#1}]}
\newcommand{\ncx}[1]{\Langle x_1,\dots,x_{#1}\Rangle}
\newcommand{\ncxsq}[1]{\Langle x_1^2,\dots,x_{#1}^2\Rangle}
\newcommand{\px}[1]{\R [x_1,\dots,x_{#1}]}
\newcommand{\ncpx}[1]{\R\!\Langle x_1,\dots,x_{#1}\Rangle}

\renewcommand{\vec}[1]{\mathbf{#1}}
\newcommand{\0}{{\color{lightgray}0}}

\title[Noncommutative Hunter's theorem]{A noncommutative generalization of Hunter's positivity theorem}

\author[S.R.~Garcia]{Stephan Ramon Garcia}
\address{Department of Mathematics and Statistics, Pomona College, 610 N. College Ave., Claremont, CA 91711} 
\email{stephan.garcia@pomona.edu}
\urladdr{\url{https://stephangarcia.sites.pomona.edu/}}

\author[J.~Vol\v{c}i\v{c}]{Jurij Vol\v{c}i\v{c}}
\address{Department of Mathematics, University of Auckland, New Zealand} 
\email{jurij.volcic@auckland.ac.nz}

\date{\today}

\keywords{Hunter's inequality, noncommutative polynomial, complete homogeneous symmetric polynomial, hermitian sum of squares}

\subjclass[2020]{05E05, 13J30, 15A45, 16S10}

\begin{document}

\begin{abstract}
Hunter proved that the complete homogeneous symmetric polynomials of even degree are positive definite.  We prove a noncommutative generalization of this result, in which the scalar variables are replaced with hermitian operators.  We provide a sharp lower bound and a sum of hermitian squares representation that are novel even in the scalar case.  
\end{abstract}

 \thanks{SRG was supported by NSF Grants DMS-2054002 and DMS-2452084.}

\maketitle

\section{Introduction}

The \emph{complete homogeneous symmetric (CHS) polynomial} of degree $d$ in $n$ (commuting) variables is the sum 
\begin{equation}\label{eq:CHS}
h_d(x_1,x_2,\ldots,x_n) := \sum_{1 \leq i_1 \leq \cdots \leq i_{d} \leq n} x_{i_1} x_{i_2}\cdots x_{i_d},
\end{equation}
of all $\binom{n+d-1}{d}$ monomials of degree $d$ in $x_1,x_2,\ldots,x_n$ \cite[Sec.~7.5]{StanleyBook2}.  For example,
$h_2(x_1,x_2)= x_1^2+x_1 x_2+x_2^2$ and
$h_4(x_1,x_2)= x_1^4  + x_1^3 x_2 + x_1^2 x_2^2 + x_1 x_2^3 + x_2^4$.
Hunter proved that CHS polynomials of even degree are nonnegative; more precisely, for $d\in\N$ and $\vec{x}= (x_1,x_2,\ldots,x_n)$ we have
\begin{equation}\label{e:HunterInequality}
h_{2d}(x_1, x_2, \ldots, x_n)\geq \frac{\| \vec{x} \|_2^{2d}}{2^d d!} ,
\end{equation} 
with equality at $\vec{x}\neq \vec{0}$ if and only if $d=2$ and $x_1+x_2+\cdots +x_n=0$ \cite{Hunter}.

Even-degree CHS polynomials give rise to prominent instances of symmetric polynomial inequalities, studied in algebraic combinatorics and real algebraic geometry \cite{Pro,Tim,CGS,BR}. 
In this paper, we explore a noncommutative analogue of CHS polynomial positivity. 
This is done from the perspective of free real algebraic geometry \cite{OHMP,HKM}, which investigates polynomial inequalities in several matrix or operator variables without dimension restrictions.
We define noncommutative CHS polynomials and prove an explicit Hunter-like lower bound for them.  Our approach yields a noncommutative sum of hermitian squares (SOHS) representation that appears novel even in the commutative case.  The noncommutative generalization of Hunter's theorem sheds new light on the classical case: we obtain better lower bounds than Hunter's in many situations while also providing sum of squares (SOS) representations for the classical CHS polynomials.

To state our results, we require some notation.
For $n\in\N$, let $\cx{n}$ denote the set of monomials in $n$ commuting variables and let $\px{n}$ denote the corresponding  real polynomial ring.  Let $\ncx{n}$ denote the set of all words in $n$ freely noncommuting variables and let $\ncpx{n}$ denote the corresponding real free $*$-algebra of noncommutative polynomials, 
where the involution $*$ is determined by $x_j^*=x_j$ for all $j=1,2,\ldots,n$.

Let $\alpha:\ncx{n}\to\cx{n}$ denote abelianization, which maps a noncommutative word to a monomial by forgetting the order.
For example, $\alpha(x_1 x_2 x_1) = x_1^2 x_2$.  A counting argument shows that
\begin{equation}\label{e:count}
\big|\alpha^{-1}(x_1^{k_1}\cdots x_n^{k_n})\big|=\binom{k_1+\cdots+k_n}{k_1,\dots,k_n}.
\end{equation}
The map $\alpha$ extends to a natural algebra homomorphism $\pi:\ncpx{n}\to \px{n}$. 
Let $\sigma:\px{n}\to\ncpx{n}$ denote the linear map
\begin{equation*}
 \sigma (m) := \frac{1}{|\alpha^{-1}(m)|}\sum_{w\in\alpha^{-1}(m)}w ,
 \qquad m \in \cx{n}.
\end{equation*}
For example, $\sigma(x_1 x_2)= \frac{1}{2}(x_1 x_2 + x_2 x_1)$.
One may view $\sigma$ as the fully symmetrized noncommutative lift.
Note that $\pi\circ\sigma=\id_{\px{n}}$. 

The \emph{noncommutative complete homogeneous symmetric (NCHS) polynomial} of degree $d$ in $n$ (noncommuting) variables is 
\begin{equation}\label{eq:Hd}
H_d(x_1,\dots,x_n) := \sigma\big(h_d(x_1,\ldots,x_n)\big) \in\ncpx{n}.
\end{equation}
For example, 
\begin{equation}\label{e:H2}
H_2(x_1,x_2) = x_1^2 + \tfrac{1}{2}(x_1 x_2 + x_2 x_1)+ x_2^2 
\end{equation}
and
\begin{equation}
\begin{split}\label{e:H4}
\hspace{-15pt} H_4(x_1,x_2)
&= x_1^4 + \tfrac{1}{4}( x_1^3 x_2 + x_1^2 x_2 x_1 + x_1 x_2 x_1^2 + x_2 x_1^3)\\
&\quad    + \tfrac{1}{6}( x_1^2 x_2^2 + x_1 x_2 x_1 x_2 + x_1 x_2^2 x_1 + x_2 x_1^2 x_2 + x_2 x_1 x_2 x_1 + x_2^2 x_1^2) \\
&\quad + \tfrac{1}{4}(x_1 x_2^3 + x_2 x_1 x_2^2 + x_2^2 x_1 x_2 + x_2^3 x_1) + x_2^4.
\end{split}
\end{equation}
Note that the definition \eqref{eq:Hd} differs from that of Gelfand et al. \cite{ncsym}, where noncommutative complete homogeneous symmetric functions are formal objects that specialize to noncommutative rational functions or non-hermitian noncommutative polynomials \cite[Section 7]{ncsym}. However, since we are interested in positivity, the definition \eqref{eq:Hd} is natural as it produces a hermitian noncommutative polynomial.
In fact, the definition \eqref{eq:Hd} fits in the framework of noncommutative symmetric functions as introduced by Rosas and Sagan \cite{RS} (our NCHS polynomials are scalar multiples of some of their noncommutative analogues of CHS polynomials).

We can now state our main result:

\begin{thm}\label{thm:main}
Let $n,d \in \N$.  
\begin{enumerate}[leftmargin=*]
    \item The noncommutative polynomial $H_{2d}(x_1,\dots,x_n)$ is a sum of $\binom{n-1+d}{d}$ hermitian squares in $\ncpx{n}$. This number of hermitian squares is minimal.

    \item For all $k\in\N$ and all hermitian operators $X_1\dots,X_n$ on a Hilbert space, \begin{equation*}
H_{2d}(X_1,\dots,X_n)\succeq\mu_{n,d} (X_1^{2d}+\cdots+X_n^{2d})    ,
\end{equation*}
in which $\succeq$ denotes the L\"owner partial order and
\begin{equation*}
\mu_{n,d}= 
\begin{cases}
1& \text{if $n=1$},\\[5pt]
\frac{\binom{n-1+2d}{2d}}{\binom{n-1+d}{d}\left(\binom{n-1+d}{d}+1\right)}& \text{if $n\ge2$ and $d$ is odd},\\[8pt]
\frac{\binom{n-1+2d}{2d}}{\binom{n-1+d}{d}\left(\binom{n-1+d}{d}+n-1\right)}&
\text{if $n\ge2$ and $d$ is even}.
\end{cases}
\end{equation*}
This lower bound is best possible.
\end{enumerate}
\end{thm}

In general, the fully symmetrized noncommutative lift does not preserve inequalities. While positivity of even-degree CHS polynomials may be seen as a special case of positivity of Schur polynomials for even partitions, the fully symmetrized noncommutative lifts of the latter are not positive semidefinite in general (Remark \ref{r:noSchur}), lending more significance to Theorem \ref{thm:main}. 
Furthermore, while the exact lower bounds for CHS polynomials are not known, the rigidity of noncommuting variables, often leveraged in free real algebraic geometry, allows us to determine the exact (and explicit) lower bounds for NCHS polynomials as in Theorem \ref{thm:main}. 
These noncommutative results also carry implications for the classical CHS polynomials, for which we obtain SOS representations and new lower bounds.

\begin{cor}\label{c:hunter}
$h_{2d}(x_1,\dots,x_n)$ is a sum of $\binom{n-1+d}{d}$ squares in $\px{n}$.
Furthermore, $h_{2d}(x_1,\dots,x_n)\ge \frac{\mu_{n,d}}{n^{d-1}} \| \vec{x} \|_2^{2d}$ for all $\vec{x}\in\R^n$.
\end{cor}

Observe that $\frac{\mu_{n,d}}{n^{d-1}}>\frac{1}{2^d d!}$ whenever $d$ is sufficiently larger than $n$, in which case Corollary \ref{c:hunter} provides a tighter estimate than \eqref{e:HunterInequality}.
Also, our method appears to be the first explicit method to express CHS polynomials as sums of squares. Another approach to SOS representations for (commutative) CHS polynomials was suggested by Speyer and Tao \cite{Tao}, although it is unclear if their approach was ultimately pursued and formally published. 

This paper is organized as follows.  Section \ref{Section:Hunter} contains the proof of Theorem \ref{thm:main}(i) (see Proposition \ref{p:middle}).  We establish Theorem \ref{thm:main}(ii) in Section \ref{Section:LowerBound} (see Proposition \ref{p:bound}).  We conclude in Section \ref{Section:Remarks} with several examples and remarks.

\section{Noncommutative Hunter's theorem}\label{Section:Hunter}

In order to study positivity of NCHS polynomials, we investigate spectral features of certain combinatorial matrices. 
We start by introducing some notation used throughout the paper.
Let $n,d\in\N$. 
By $\cx{n}_d\subset\cx{n}$ and $\ncx{n}_d\subset\cx{n}$ we denote the sets of monomials and words of degree $d$, respectively. Similarly, let $\px{n}_d\subset \px{n}$ and $\ncpx{n}_d\subset \ncpx{n}$ denote the subspaces of homogeneous polynomials of degree $d$.
When used as index sets, we endow $\cx{n}_d$ and $\ncx{n}_d$ with lexicographic order. 
Define the rational matrices
\begin{equation}\label{e:G}
\widetilde G_{n,d}=\left[ \frac{1}{|\alpha^{-1}(uv)|} \right]_{u,v\in \cx{n}_d},
\qquad
G_{n,d}=\left[ \frac{1}{|(\alpha^{-1}\circ\pi)(u^*v)|} \right]_{u,v\in \ncx{n}_d}\!\!\!.
\end{equation}
By \eqref{e:count}, their entries are reciprocals of certain multinomial coefficients.
Throughout the paper, we often interpret matrices \eqref{e:G} as linear maps $\widetilde{G}_{n,d}:\px{n}_d\to\px{n}_d$ and $G_{n,d}:\ncpx{n}_d\to\ncpx{n}_d$ in a natural way.
For our purposes, the crucial properties of $\widetilde{G}_{n,d}$ and $G_{n,d}$ are the following.

\begin{prop}\label{p:G}
Let $n,d\in\N$.
\begin{enumerate}[leftmargin=*]
\item $\widetilde{G}_{n,d}$ is positive definite.
\item $G_{n,d}$ is positive semidefinite, 
$\ker G_{n,d} = \ker\pi|_{\ncpx{n}\!{}_d}$, and 
$\rk G_{n,d}= \binom{n-1+d}{d}$.
\end{enumerate}
\end{prop}

\begin{proof}
Define a linear functional
$\lambda:\px{n}\to\R$ by
\begin{equation*}
\lambda(p)=\int_{\Sigma_{n-1}} p\,{\rm d}\Sigma_{n-1},    
\end{equation*}
where $\Sigma_{n-1}=\{(t_1,\dots,t_n)\in\R_{\ge0}^n\colon t_1+\cdots+t_n=1\}$ is the standard $(n-1)$-simplex endowed with the Lebesgue measure. 
Note that $\lambda$ is nonnegative on nonnegative polynomials, and in particular on squares.
By the Dirichlet integral formula (see, e.g., the Dirichlet or multivariate beta distribution \cite[Section 3.2.5]{bayesian}) and \eqref{e:count},
\begin{equation}\label{e:int}
\begin{split}
\lambda(x_1^{k_1}\cdots x_n^{k_n})&=\sqrt{n}\frac{k_1!\cdots k_n!}{(k_1+\cdots+k_n+n-1)!}\\
&=\frac{\sqrt{n}(k_1+\cdots+k_n)!}{(k_1+\cdots+k_n+n-1)!}\cdot
\frac{1}{\binom{k_1+\cdots+k_n}{k_1,\dots,k_n}}\\
&=\frac{\sqrt{n}(k_1+\cdots+k_n)!}{(k_1+\cdots+k_n+n-1)!}\cdot
\frac{1}{|\alpha^{-1}(x_1^{k_1}\cdots x_n^{k_n})|}.
\end{split}
\end{equation}
If $p\in\px{n}$ is homogeneous and $\lambda(p^2)=0$, then $p|_{\Sigma_{n-1}}\equiv0$. Since $p$ vanishes on $\Sigma_{n-1}$ if and only if it is a multiple of $x_1+\cdots+x_n-1$, homogeneity implies that $p=0$.
Thus, $(p,q)\mapsto \lambda(pq)$ is an inner product on $\px{n}_d$. 

\medskip\noindent(i): By \eqref{e:int}, the matrix $\widetilde{G}_{n,d}$ is a positive multiple of the Gram matrix of the linearly independent vectors $\cx{n}_d$ with respect to an inner product. Thus, $\widetilde{G}_{n,d}$ is positive definite.

\medskip\noindent(ii): 
The functional $\lambda\circ\pi:\ncpx{n}\to\R$ is nonnegative on hermitian squares, and thus gives rise to a semi-inner product on $\ncpx{n}_d$. The matrix $G_{n,d}$ is a positive multiple of the Gram matrix of the linearly independent vectors $\ncx{n}_d$ with respect to it, and thus positive semidefinite. If $f\in \ncpx{n}$ is homogeneous, then $(\lambda\circ\pi)(f^*f)=0$ if and only if $\pi(f)=0$. Thus, $\ker G_{n,d}=\ker \pi|_{\ncpx{n}_d}$ and $\rk G_{n,d}=\rk \widetilde{G}_{n,d}=\binom{n-1+d}{d}$.
\end{proof}

Now equipped with Proposition \ref{p:G}, we are ready to establish Theorem \ref{thm:main}(i).

\begin{prop}\label{p:middle}
Let $n,d\in\N$. In $\ncpx{n}$, we have
\begin{equation}\label{e:middle}
H_{2d}(x_1,\dots,x_n)=\vec{w}^*\, G_{n,d}\, \vec{w},
\end{equation}
where $\vec{w}$ is the column vector of words in $\ncx{n}_{d}$ ordered lexicographically. The noncommutative polynomial $H_{2d}(x_1,\dots,x_n)$ is a sum of $\binom{n-1+d}{d}$ hermitian squares in $\ncpx{n}$, and this number of hermitian squares is minimal.
\end{prop}

\begin{proof}
Each element in $\ncx{n}_{2d}$ can be uniquely written as $u^*v$ for some $u,v\in\ncx{n}_{d}$. Thus, every (hermitian) $f\in\ncpx{n}_{2d}$ can be written as $f=\vec{w}^* A\vec{w}$ for a unique (symmetric) real matrix $A$; namely, the coefficient of $u^*v$ in $f$ is the $(u,v)$ entry of $A$.
By \eqref{eq:Hd}, the coefficient of $u^*v$ in $H_{2d}(x_1,\dots,x_n)$ is $\frac{1}{|\alpha^{-1}(\pi(u^*v))|}$. Thus, \eqref{e:middle} holds by the definition of $G_{n,d}$.
Since $G_{n,d}$ is positive semidefinite, it factors as $G_{n,d}=S^*S$, where $S$ has $\rk G_{n,d} = \binom{n-1+d}{d}$ rows. Hence, $H_{2d}(x_1,\dots,x_n)=(S\vec{w})^* (S\vec{w})$ is a sum of $\binom{n-1+d}{d}$ hermitian squares.
To see that this number of hermitian squares is minimal, suppose $H_{2d}(x_1,\dots,x_n)=\sum_{j=1}^r q_j^*q_j$ for $q_j\in \ncpx{n}$. The homogeneity of $H_{2d}(x_1,\dots,x_n)$ ensures that $q_j\in\ncpx{n}_{d}$. Expanding $q_j$ along the basis $\ncx{n}_{d}$ gives $H_{2d}(x_1,\dots,x_n)=\vec{w}^* Q^*Q \vec{w}$ for a matrix $Q$ with $r$ rows. By the uniqueness observation from the start of the proof, we have $Q^*Q= G_{n,d}$, and thus $r\ge\binom{n-1+d}{d}$.
\end{proof}

\begin{rem}
One can write $G_{n,d}=S^* \Lambda S$ with $\Lambda =\big[ \begin{smallmatrix} \lambda_1 & \\[-8pt] & \!\!\ddots \end{smallmatrix}\big]$ diagonal using only linear operations over $\Q$ \cite[Corollary I.2.4]{lam}. This gives rise to a weighted SOHS representation of $H_{2d}(x_1,\dots,x_n)$ with rational coefficients: if the columns of $S$ are indexed by $\ncx{n}_d$, then
\begin{equation*}
H_{2d}(x_1,\dots,x_n)=\sum_j \lambda_{j}\, s_j^*s_j,\qquad s_j=\sum_{w\in\ncx{n}_{d}} S_{j,w}\, w.
\end{equation*}
\end{rem}

\begin{exa}
Let $n=2$ and $d=1$.  Then $\vec{w} = \big[ \begin{smallmatrix} x_1 \\ x_2 \end{smallmatrix}\big]$ and $G_{2,1}=\big[ \begin{smallmatrix} 1 & \frac{1}{2} \\ \frac{1}{2} & 1 \end{smallmatrix} \big] = S^* \Lambda S$, in which $S = \big[\begin{smallmatrix}1 & \frac{1}{2} \\ 0 & 1 \end{smallmatrix}\big]$ and $\Lambda = \operatorname{diag}(1,\frac{3}{4})$, so $S \vec{w} = \big[ \begin{smallmatrix} x_1 + \frac{1}{2}x_2\\ x_2 \end{smallmatrix}\big]$ and hence
\begin{equation*}
H_2(x_1,x_2)
= (x_1 + \tfrac{1}{2}x_2)^*(x_1 + \tfrac{1}{2}x_2) + \tfrac{3}{4}x_2^*x_2,
\end{equation*}
which reduces to \eqref{e:H2}.  If $x_1$ and $x_2$ commute, we recover $h_2(x_1,x_2) = x_1^2 + x_1 x_2 + x_2^2$.
\end{exa}

\begin{exa}
Let $n=2$ and $d=2$.  Then $G_{2,2} = S^* \Lambda S$, in which
\begin{equation}\label{e:G22}
    G_{2,2}
    =
    \begin{bmatrix}
        1 & \frac{1}{4} & \frac{1}{4} & \frac{1}{6} \\[2pt]
        \frac{1}{4} & \frac{1}{6} & \frac{1}{6} & \frac{1}{4}\\[2pt]
        \frac{1}{4} & \frac{1}{6} & \frac{1}{6} & \frac{1}{4}\\[2pt]
        \frac{1}{6} & \frac{1}{4} & \frac{1}{4} & 1
    \end{bmatrix},
    \quad
    S = 
    \begin{bmatrix}
 1 & \frac{1}{4} & \frac{1}{4} & \frac{1}{6} \\[2pt]
 \0 & \0 & \0 & 1 \\
\0 & 1 & 1 & 2 \\
  \0 & \0 & \0 & 1 \\
    \end{bmatrix},
    \quad \text{and} \quad
    \Lambda = 
    \begin{bmatrix}
 1 & \0 & \0 & \0 \\
 \0 & \frac{5}{9} & \0 & \0 \\
 \0 & \0 & \frac{5}{48} & \0 \\
 \0 & \0 & \0 & 0 \\
    \end{bmatrix}.    
\end{equation}
Thus, $H_4(x_1,x_2) = s_1^*s_1 + \frac{5}{9}s_2^*s_2 + \frac{5}{48} s_3^* s_3$, in which $s_1= x_1^2 + \frac{1}{4} x_1 x_2 + \frac{1}{4} x_2 x_1 + \frac{1}{6} x_2^2$, $s_2 = x_2^2$, and $s_3 = x_1 x_2 + x_2 x_1 + 2 x_2^2$.
If $x_1$ and $x_2$ commute, we recover 
$h_4(x_1,x_2) = x_1^4+ x_1^3 x_2 + x_1^2 x_2^2 + x_1 x_2^3 + x_2^4$.
\end{exa}

There is a reason we focus on NCHS polynomials and a noncommutative version of Hunter's theorem, instead of analogues for more general Schur polynomials:

\begin{rem}\label{r:noSchur}
In the commutative landscape, nonnegativity of CHS polynomials can be viewed as a special case of nonnegativity of Schur polynomials for even partitions.
To every partition $\lambda\vdash d$ with $n$ nonnegative parts one assigns the Schur polynomial $s_\lambda(x_1,\dots,x_n)$ (see \cite[Section 7.10]{StanleyBook2} for the definition). Then, $h_d(x_1,\dots,x_n)=s_{(d,0,\dots,0)}(x_1,\dots,x_n)$. In general, if all the parts of $\lambda$ are even, the polynomial $s_\lambda(x_1,\dots,x_n)$ is nonnegative on $\R^n$ (see Speyer's response \cite{Tao} using the bialternant formula in terms of Vandermonde-like determinants \cite[Section 7.15]{StanleyBook2} and Descartes' rule of signs). However, while the noncommutative polynomial $\sigma(h_{2d}(x_1,\dots,x_n))$ is globally positive semidefinite by Proposition \ref{p:middle}, this is no longer true for $\sigma(s_\lambda (x_1,\dots,x_n))$ with a general even $\lambda$.

For example, $s_{(2,2)}(x_1,x_2)=x_1^2 x_2^2$, and thus
$$\sigma\left(s_{(2,2)}(x_1,x_2)\right)=\frac16\left(
x_1^2x_2^2+x_1x_2x_1x_2+x_1x_2^2x_1+x_2x_1^2x_2+x_2x_1x_2x_1+x_2^2x_1^2
\right).$$
This noncommutative polynomial evaluates to $\frac16[\begin{smallmatrix}1&2\\2&3\end{smallmatrix}]\not\succeq0$ at the pair of hermitian matrices $[\begin{smallmatrix}0&0\\0&1\end{smallmatrix}],[\begin{smallmatrix}2&1\\1&0\end{smallmatrix}]$. 
This agrees with the heuristic that fully symmetrized noncommutative lifts of nonnegative polynomials are usually not positive semidefinite, and this somewhat distinguishes CHS polynomials from general Schur polynomials.
\end{rem}

\section{Sharp lower bound}\label{Section:LowerBound}

In order to establish an exact quantitative version of positivity for NCHS polynomials, we require two further matrices.
For $n,d\in\N$, let $\widetilde{M}_{n,d}$ and $M_{n,d}$ be square matrices indexed by $\cx{n}_d$ and $\ncx{n}_d$, respectively, whose $(x_j^d,x_j^d)$ entries for $j=1,2,\dots,n$ are 1, and all the other entries are 0.
Both $\widetilde{M}_{n,d}$ and $M_{n,d}$ are diagonal projections of rank $n$. 
The next lemma indicates the role of $M_{n,d}$ in estimating the NCHS polynomials.

\begin{lem}\label{l:geneig}
Let $n,d\in\N$ and $\mu\in \R$. The following are equivalent.
\begin{enumerate}[leftmargin=*]
\item $H_{2d}(X_1,\dots,X_n)\succeq\mu (X_1^{2d}+\cdots+X_n^{2d})$ for all hermitian operators $X_1,\dots,X_n$ on a Hilbert space.
\item $H_{2d}(X_1,\dots,X_n)\succeq\mu (X_1^{2d}+\cdots+X_n^{2d})$ for all symmetric $K\times K$ matrices $X_1,\dots,X_n$, where $K=\frac{n^{d+1}-1}{n-1}$.
\item $G_{n,d}-\mu M_{n,d}$ is positive semidefinite.
\end{enumerate}
\end{lem}

\begin{proof}
Clearly, (i) $\Rightarrow$ (ii).
Let $f_\mu = H_{2d}(x_1,\dots,x_n)- \mu (x_1^{2d}+\cdots+x_n^{2d})$; then
$f_\mu = \vec{w}^* (G_{n,d}-\mu M_{n,d})\vec{w}$, where $\vec{w}$ is the column vector of words in $\ncx{n}_d$. Thus, (iii) $\Rightarrow$ (i). If $f_\mu$ is positive semidefinite on all symmetric $K\times K$ matrices, then it is a sum of hermitian squares in $\ncpx{n}$ \cite[Theorem 1.1 and Remark 1.2]{MP05}. Since $f_\mu$ is homogeneous, the corresponding hermitian squares are homogeneous, so $f_\mu=\vec{w}^* P\vec{w}$ for some $P\succeq0$. By homogeneity, $G_{n,d}-\mu M_{n,d}=P$, so (ii) $\Rightarrow$ (iii).
\end{proof}

For $n,d\in\N$, let us define $\mu_{n,d}$ as the largest $\mu\in\R$ that satisfies the equivalent statements in Lemma \ref{l:geneig},
\begin{equation}\label{e:mu}
\begin{split}
\mu_{n,d}&=\max\left\{\mu\in\R\colon 
H_{2d}(X_1,\dots,X_n)\succeq\mu (X_1^{2d}+\cdots+X_n^{2d})\text{ for all }X_j
\right\}\\
&=\max\{\mu\in\R\colon G_{n,d}-\mu M_{n,d}\succeq0\}.
\end{split}
\end{equation}
The second line in \eqref{e:mu} justifies the use of maximum instead of supremum (also by compactness, since it suffices to restrict to $K\times K$ contractions $X_j$ by Lemma \ref{l:geneig}(ii) and homogeneity).
The number $\mu_{n,d}$, nonnegative by Proposition \ref{p:middle}, is the largest lower bound on $H_{2d}(x_1,\dots,x_n)$ with respect to the noncommutative positive form $x_1^{2d}+\cdots+x_n^{2d}$.
Clearly, $\mu_{1,d}=1$.
As an auxiliary step towards a closed-form expression for $\mu_{n,d}$, we require some partial information about the inverse of $\widetilde G_{n,d}$ (which exists by Proposition \ref{p:G}(i)). Using combinatorial means, we calculate the preimage of $x_1^d$ under $\widetilde G_{n,d}$ as follows.

\begin{prop}\label{p:inverse}
Let $n,d\in\N$. The map $\widetilde{G}_{n,d}^{-1}:\px{n}_d\to\px{n}_d$ sends $x_1^d$ to
\begin{equation}\label{e:inverse}
\frac{\binom{n-1+d}{d}}{\binom{n-1+2d}{2d}} \sum_{i=0}^d (-1)^{d-i}\binom{d}{i}\binom{n-1+d}{i} x_1^i(x_2+\cdots+x_n)^{d-i}.
\end{equation}
\end{prop}

\begin{proof}
Write the polynomial above as $\sum_{v\in\cx{n}_d}c_v v$.
We want to show that
\begin{equation*}
\sum_{v\in\cx{n}_d} \big(\widetilde G_{n,d})_{u,v}\cdot c_v
=
\begin{cases}
1 & \text{if $u=x_1^d$},\\
0 & \text{otherwise}.
\end{cases}
\end{equation*}
That is, for nonnegative integers $\ell_1,\dots,\ell_n$ with $\ell_1+\cdots+\ell_n=d$, we claim that
\begin{equation}\label{e:todo}
\sum_{k_1+\cdots+k_n=d} 
(-1)^{d-k_1}\binom{d}{k_1}\binom{n-1+d}{k_1}\frac{(d-k_1)!}{k_2!\cdots k_n!} 
\frac{(\ell_1+k_1)!\cdots (\ell_n+k_n)!}{(2d)!}
\end{equation}
equals $\frac{\binom{n-1+2d}{2d}}{\binom{n-1+d}{d}}$ if $\ell_1=d$ and $0$ otherwise.
We can rewrite \eqref{e:todo} as
\begin{align*}
&\frac{d!}{(2d)!}\sum_{k_1+\cdots+k_n=d} 
(-1)^{d-k_1}\binom{n-1+d}{k_1}\frac{(\ell_1+k_1)!\cdots (\ell_n+k_n)!}{k_1!\cdots k_n!} \\
&\quad= \frac{d!}{(2d)!}\sum_{k_1=0}^d 
(-1)^{d-k_1}\binom{n-1+d}{k_1}\frac{(\ell_1+k_1)!}{k_1!} \!\!\!\!\!\! \sum_{k_2+\cdots+k_n=d-k_1} \!\!\!\!\!\! \frac{(\ell_2+k_2)!\cdots (\ell_n+k_n)!}{k_2!\cdots k_n!}.
\end{align*}
By the binomial coefficient reflection $\binom{\ell+k}{k}=(-1)^k\binom{-\ell-1}{k}$ \cite[(5.14) on page 164]{concrete} and generalized Vandermonde's convolution \cite[(5.27) on page 170 and Exercise 5.62 on page 248]{concrete},
\begin{align*}
&\sum_{k_2+\cdots+k_n=d-k_1}\frac{(\ell_2+k_2)!\cdots (\ell_n+k_n)!}{k_2!\cdots k_n!}\\
&\quad =\ell_2!\cdots\ell_n!\sum_{k_2+\cdots+k_n=d-k_1}\binom{\ell_2+k_2}{k_2}\cdots \binom{\ell_n+k_n}{k_n}\\
&\quad=\ell_2!\cdots\ell_n!(-1)^{d-k_1}\sum_{k_2+\cdots+k_n=d-k_1}\binom{-\ell_2-1}{k_2}\cdots \binom{-\ell_n-1}{k_n}\\
&\quad =\ell_2!\cdots\ell_n!(-1)^{d-k_1}\binom{-\ell_2-\cdots-\ell_n-(n-1)}{d-k_1}.
\end{align*}
Thus, \eqref{e:todo} becomes
\begin{align*}
&\frac{d!}{(2d)!}\sum_{k_1=0}^d 
(-1)^{d-k_1}\binom{n-1+d}{k_1}\frac{(\ell_1+k_1)!}{k_1!}\ell_2!\cdots\ell_n!(-1)^{d-k_1}\binom{-\ell_2-\cdots-\ell_n-(n-1)}{d-k_1}  \\
&\quad =  \frac{d! \ell_1!\cdots\ell_n!}{(2d)!}\sum_{k_1=0}^d 
\binom{n-1+d}{k_1}\binom{\ell_1+k_1}{k_1}\binom{\ell_1-d-n+1}{d-k_1}.
\end{align*}
Setting $m=n-1+d$, it suffices to prove
\begin{equation}\label{e:todo1}
\sum_{k=0}^d \binom{m}{k}\binom{\ell+k}{k}\binom{\ell-m}{d-k}
=
\begin{cases}
\binom{m+d}{d} & \text{if $\ell=d$},\\
0 & \text{if $\ell<d$}.    
\end{cases}
\end{equation}
This is a consequence of the following calculation, valid for a general $\ell\in\Z$, 
\begin{align*}
&\sum_{k=0}^d \binom{m}{k}\binom{\ell-m}{d-k}\binom{\ell+k}{k}
=\sum_{k=0}^d \binom{m}{k}\binom{\ell-m}{d-k}\sum_{j=0}^d\binom{\ell}{j}\binom{k}{j}\\
=\,&\sum_{j=0}^d\sum_{k=j}^d \binom{\ell-m}{d-k}\binom{\ell}{j}\binom{m}{k}\binom{k}{j}
=\sum_{j=0}^d\sum_{k=j}^d \binom{\ell-m}{d-k}\binom{\ell}{j}\binom{m}{j}\binom{m-j}{k-j}\\
=\,&\sum_{j=0}^d \binom{\ell}{j}\binom{m}{j}\sum_{k=j}^d\binom{\ell-m}{d-k}\binom{m-j}{k-j}
=\sum_{j=0}^d \binom{\ell}{j}\binom{m}{j}\sum_{i=0}^{d-j}\binom{\ell-m}{d-j-i}\binom{m-j}{i}\\
=\,&\sum_{j=0}^d \binom{m}{j}\binom{\ell}{j}\binom{\ell-j}{d-j}
=\sum_{j=0}^d \binom{m}{j}\binom{\ell}{d}\binom{d}{j}=\binom{\ell}{d}\binom{m+d}{d},
\end{align*}
where we used Vandermonde's convolution \cite[(5.27) on page 170]{concrete} twice.
\end{proof}

Since $\mu_{n,d}$ is the optimal lower bound for $H_{2d}(x_1,\dots,x_n)$ by the definition \eqref{e:mu}, the following derivation of a closed formula for $\mu_{n,d}$ completes the proof of Theorem \ref{thm:main}(ii).

\begin{prop}\label{p:bound}
Let $n,d\in\N$ with $n\ge2$. Then
\begin{equation*}
\mu_{n,d}= 
\frac{\binom{n-1+2d}{2d}}{\binom{n-1+d}{d}\left(\binom{n-1+d}{d}+\Delta\right)},
    \end{equation*}
where $\Delta=1$ if $d$ is odd and $\Delta=n-1$ if $d$ is even.
\end{prop}

\begin{proof}
For simplicity, we suppress the subscripts in $G_{n,d},M_{n,d},\widetilde G_{n,d},\widetilde{M}_{n,d}$.
By Lemma \ref{l:geneig},
\begin{equation}\label{e:eig1}
\mu_{n,d}=\max\{\mu\in\R\colon G-\mu M\succeq0\}.
\end{equation}
We first claim that $\ker G\subseteq \ker M$. Let $f\in\ker G$. Then $\pi(f)=0$ by Proposition \ref{p:G}. Since $\pi$ is abelianization, no term of $f$ is a scalar multiple of $x_j^d$. Thus, $f\in\ker M$ by the definition of $M$.

Since $\ker G\subseteq \ker M$, we can consider \eqref{e:eig1} modulo $\ker G= \ker \pi\cap\ncpx{n}_d$. That is, it suffices to work with $\widetilde G$ and $\widetilde{M}$ instead of $G$ and $M$, respectively. Hence, 
\begin{equation}\label{e:eig2}
\mu_{n,d}=\max\{\mu\in\R\colon \widetilde G-\mu \widetilde{M}\succeq0\}.
\end{equation}
Recall that $\widetilde G$ is positive definite by Proposition \ref{p:G} and $\widetilde{M}$ is the diagonal projection onto $\operatorname{span}\{x_1^{d},\dots,x_n^{d}\}$. 
Let us reorder $\cx{n}_d$ so that it starts with $x_1^{d},\dots,x_n^{d}$, and write
\begin{equation*}
\widetilde{G}=\begin{bmatrix} A& B^*\\B& C\end{bmatrix}
\quad \text{and} \quad
\widetilde{M}=\begin{bmatrix} I& 0\\0& 0\end{bmatrix}   
\end{equation*}
with respect to $\{x_1^{d},\dots,x_n^{d}\}$ and $\cx{n}_d\setminus \{x_1^{d},\dots,x_n^{d}\}$.
By \eqref{e:eig2}, 
\begin{align*}
\mu_{n,d}
&=\max\left\{\mu\in\R\colon\begin{bmatrix}
A-\mu I& B^* \\ B&C
\end{bmatrix}\succeq0\right\} \\
&=\max\left\{\mu\in\R\colon (A-B^*C^{-1}B)-\mu I\succeq0\right\},    
\end{align*}
so $\mu_{n,d}$ is the smallest eigenvalue of the Schur complement $A-B^*C^{-1}B$.
Note that $G':=(A-B^*C^{-1}B)^{-1}$ is the principal block of $\widetilde G^{-1}$ indexed by $x_1^d,\dots,x_n^d$.
Since $\widetilde G$ is invariant under permutations of $x_1,\dots,x_n$ in its row and column indices simultaneously, its inverse $\widetilde G^{-1}$ displays the same invariance. Hence, all the off-diagonal entries of $G'$ are the same, and all the diagonal entries of $G'$ are the same. 
By Proposition \ref{p:inverse}, the diagonal and the off-diagonal terms of $G'$ are
\begin{equation}\label{e:rho}
\rho_0:=\frac{\binom{n-1+d}{d}^2}{\binom{n-1+2d}{2d}} \qquad\text{and}\qquad 
\rho_1:=(-1)^d\frac{\binom{n-1+d}{d}}{\binom{n-1+2d}{2d}},
\end{equation}
respectively. Indeed, Proposition \ref{p:inverse} encodes the first column of $\widetilde G^{-1}$, as follows. If $p$ is the polynomial that maps to $x_1^d$ under $\widetilde G$, then $\rho_0$ (the $(1,1)$ entry of $\widetilde G^{-1}$) is the coefficient of $x_1^d$ in $p$, and $\rho_1$ (the $(2,1)$ entry of $\widetilde G^{-1}$) is the coefficient of $x_2^d$ in $p$.
The spectrum of $G'$ is therefore $\{\rho_0+(n-1)\rho_1,\rho_0-\rho_1\}$, so the spectrum of $A-B^*C^{-1}B$ is $\{(\rho_0+(n-1)\rho_1)^{-1},(\rho_0-\rho_1)^{-1}\}$. Thus,
\begin{equation*}
\mu_{n,d}= 
\begin{cases}
\frac{1}{\rho_0+(n-1)\rho_1}& \text{if $d$ is even},\\
\frac{1}{\rho_0-\rho_1}&\text{if $d$ is odd},    
\end{cases}
\end{equation*}
which gives the desired formula.
\end{proof}

In the spirit of Hunter's theorem, here is the best lower bound valid for every $n$.

\begin{cor}
Let $d\in\N$. For all $n\in\N$ and hermitian operators $X_1,\dots,X_n$,
$$H_{2d}(X_1,\dots,X_n)\succeq \frac{1}{\binom{2d}{d}} \left(X_1^{2d}+\cdots+X_n^{2d}\right),$$
and this is the optimal $n$-independent bound.
\end{cor}

\begin{proof}
The sequence of optimal bounds $(\mu_{n,d})_n$ is monotonically decreasing, and the formula in Proposition \ref{p:bound} implies
\begin{equation*}
\lim_{n\to\infty}\mu_{n,d}=\lim_{n\to\infty} \frac{\binom{n+2d}{2d}}{\binom{n+d}{d}^2}
=\frac{(d!)^2}{(2d)!}\lim_{n\to\infty} \frac{(n+2d)^{\underline{2d}}}{\left((n+d)^{\underline{d}}\right)^2}
=\frac{(d!)^2}{(2d)!} = \binom{2d}{d}^{-1},
\end{equation*}
where an underscored exponent denotes a falling factorial.
\end{proof}

\section{Remarks}\label{Section:Remarks}

We conclude the paper with certain subtle aspects of exact lower bounds for NCHS polynomials, the consequences of Theorem \ref{thm:main} for positivity of the classical CHS polynomials, and two open problems.

Proposition \ref{p:bound} gives the sharp lower bound on $H_{2d}(x_1,\dots,x_n)$ in terms of $x_1^{2d}+\cdots+x_n^{2d}$. In particular, since $\mu_{n,d}>0$, it implies that $H_{2d}(x_1,\dots,x_n)$ is a definite form, in the sense that
$$\ker H_{2d}(X_1,\dots,X_n)=\ker X_1\cap\cdots\cap \ker X_n$$
for all hermitian operators $X_j$.
The next lemma illustrates that in the noncommutative context, not all forms are suitable for bounding $H_{2d}(x_1,\dots,x_n)$ from below.

\begin{prop}\label{p:nobound}
Let $n,d\in\N$, with $n\ge 2$ and  $d\ge 3$.
There is no $\beta>0$ such that
$H_{2d}(X_1,\dots,X_n)\succeq\beta (X_1^2+\cdots+X_n^2)^d$ is positive semidefinite for all symmetric $K\times K$ matrices $X_1,\dots,X_n$, where $K=\frac{n^{d+1}-1}{n-1}$.
\end{prop}

\begin{proof}
Let $\ncxsq{n}_d\subset \ncx{n}_{2d}$ denote the set of words in $x_1^2,\dots,x_n^2$. Observe that
\begin{equation}\label{e:power}
(x_1^2+\cdots+x_n^2)^d = \sum_{w\in\ncxsq{n}_d} w.
\end{equation}
Let $B_{n,d}$ be the square matrix indexed by $\ncx{n}_d$ and defined by
\begin{equation*}
(B_{n,d})_{u,v} =
\begin{cases}
1 & \text{if $u^*v\in \ncxsq{n}_d$}, \\
0 & \text{otherwise}.    
\end{cases}
\end{equation*}
Then $B_{n,d}$ is positive semidefinite; as with $G_{n,d}$, we may view it as a linear map $\ncpx{n}_d\to\ncpx{n}_d$.

Let $\beta>0$. Then $H_{2d}(x_1,\dots,x_n)-\beta (x_1^2+\cdots+x_n^2)^d=\vec{w}^* (G_{n,d}-\beta B_{n,d})\vec{w}$ by \eqref{e:power}, where $\vec{w}$ is the column vector of words in $\ncx{n}_d$. As in the proof of Lemma \ref{l:geneig} we see using \cite[Theorem 1.1 and Remark 1.2]{MP05} that $H_{2d}(x_1,\dots,x_n)-\beta (x_1^2+\cdots+x_n^2)^d$ is positive semidefinite on all $n$-tuples of symmetric $K\times K$ matrices if and only if $G_{n,d}-\beta B_{n,d}\succeq 0$. Now let $f=x_1^{d-3}(x_1x_2^2-x_2^2x_1)\in\ncpx{n}_d$. Then $\pi(f)=0$ and so $f\in\ker G_{n,d}$ by Proposition \ref{p:G}, yet $f\notin\ker B_{n,d}$. Hence, $G_{n,d}-\beta B_{n,d} \not \succeq 0$.
\end{proof}

Next, we specialize Theorem \ref{thm:main} to scalar variables, which gives rise to a new lower bound for classical CHS polynomials.

\begin{cor}\label{c:hunter2}
Let $d,n\in\N$. For every $\vec{x}\in\R^n$,
\begin{equation}\label{e:HunterImproved}
h_{2d}(\vec{x})\geq \frac{\mu_{n,d}}{n^{d-1}} \| \vec{x}\|_2^{2d}.
\end{equation}
\end{cor}

\begin{proof}
On one hand, $h_{2d}(\vec{x})\ge \mu_{n,d}\| \vec{x}\|_{2d}^{2d}$ by Theorem \ref{thm:main}. On the other hand, $\|\vec{x}\|_2^{2d} \le n^{d-1}\|\vec{x}\|_{2d}^{2d}$ by a standard estimate between $\ell^2$ and $\ell^{2d}$ norms on $\R^n$.
\end{proof}

\begin{rem}
Let us compare the estimate from Corollary \ref{c:hunter2} with the estimate in Hunter's theorem $h_{2d}(\vec{x})\ge \frac{1}{2^d d!}\| \vec{x}\|_2^{2d}$.
For a fixed $n$, the factor $\frac{\mu_{n,d}}{n^{d-1}}$ decays exponentially in $d$ (since the binomial coefficients in $\mu_{n,d}$ are polynomials in $d$), while the factor $\frac{1}{2^d d!}$ decays factorially in $d$. Thus, 
$$\frac{\mu_{n,d}}{n^{d-1}}\gg \frac{1}{2^d d!}\qquad \text{when}\quad d\gg n,$$
so as $d\to\infty$, our lower bound on $h_{2d}(x_1,\dots,x_n)$ is ultimately better than the original lower bound provided by Hunter's theorem.
\end{rem}

When $d=1$, the identity $H_2(x_1,\dots,x_n)-\frac12(x_1^2+\cdots+x_n^2)=\frac12 (x_1+\cdots+x_n)^2$ shows that $\mu_{n,1}$ is also the optimal bound for $h_2(x_1,\dots,x_n)$, and tightness is attained on the hyperplane $x_1+\cdots+x_n=0$. Analogous conclusions fail when $d=2$, as shown by the following example.

\begin{exa}\label{exa22}
We demonstrate explicitly some subtleties of $\mu_{n,d}$ in the case $n=d=2$.
First, while the constants $\mu_{n,d}$ are optimal for NCHS polynomials, they are not tight for CHS polynomials.
Concretely, Lagrange multipliers show that
$$\min_{x_1^4+x_2^4=1}h_4(x_1,x_2)=\frac12,$$
so $h_4(x_1,x_2)\ge \frac12 (x_1^4+x_2^4)$ for all $(x_1,x_2)\in\R^2$, and $\frac12$ is optimal. 

Secondly, since $\mu_{2,2}=\frac{5}{12}<\frac12$, Proposition \ref{p:bound} establishes that $H_4(X_1,X_2)\succeq \frac{5}{12} (X_1^4+X_2^4)$ for all hermitian operators $X_1,X_2$, and $\frac{5}{12}$ is optimal. To see the last assertion in a more direct way, consider the one-parametric family of matrix pairs
\begin{equation*}
X_1(t)=\phi t\begin{bmatrix}
1 & 1\\ 1& \frac{\sqrt{2t^{-4}-61}-3}{2\phi^2}
\end{bmatrix}
\quad \text{and} \quad
X_2(t)=-\phi^{-1}t \begin{bmatrix}
1 & 1\\ 1& \frac{\sqrt{2t^{-4}-61}-3}{2\phi^{-2}}
\end{bmatrix}
\end{equation*}
for $t\in(0,\sqrt[4]{\frac{2}{61}})$, where $\phi=\frac{1+\sqrt{5}}{2}$. The $(1,1)$ entry of $X_1(t)^4+X_2(t)^4$ is $1$, and the $(1,1)$ entry of $h_4(X_1(t),X_2(t))$ is $\frac{5}{12}(1+25t^4)$. Thus, the $(1,1)$ entry of
$$h_4\big(X_1(t),X_2(t)\big)-\tfrac{5}{12} \left(X_1(t)^4+X_2(t)^4\right)$$
goes to $0$ as $t\to 0$, demonstrating that there is no better bound than $\frac{5}{12}$.

The reason behind demonstrating optimality of $\frac{5}{12}$ with a family of matrix pairs instead of a single matrix pair is the following claim: for every pair of symmetric $k\times k$ matrices $X_1,X_2$ there exists $\ve>0$ such that $H_4(X_1,X_2)\succeq (\frac{5}{12}+\ve) (X_1^4+X_2^4)$. That is, for each pair there exists a better bound than $\frac{5}{12}$, but not a better one that would hold for all pairs; note that this does not contradict the fact that $\mu_{2,2}$ is the maximum (not just the supremum) as in \eqref{e:mu}.
To prove the above claim, we first observe that
\begin{equation}\label{e:ker}
\ker \left(H_4(X_1,X_2)-\frac{5}{12}(X_1^4+X_2^4)\right) = \ker X_1\cap \ker X_2.
\end{equation}
Indeed, the SOHS representation
\begin{equation*}
H_4(x_1,x_2)-\frac{5}{12}(x_1^4+x_2^4) = \frac{1}{24} s_1^*s_1+\frac{5}{24} s_2^*s_2 ,
\end{equation*}
where
\begin{equation*}
s_1=x_1^2+x_2^2+2(x_1+x_2)^2 \quad\text{and}\quad s_2=x_1^2-x_2^2,
\end{equation*}
shows that $H_4(X_1,X_2)\vec{v}=\vec{0}$ implies $s_1(X_1,X_2)\vec{v}=\vec{0}$, and then $X_1\vec{v}=X_2\vec{v}=\vec{0}$.
If $\ve>0$ is the ratio between the smallest positive eigenvalue of $H_4(X_1,X_2)-\frac{5}{12}(X_1^4+X_2^4)$ and the largest eigenvalue of $X_1^4+X_2^4$, then \eqref{e:ker} implies that $H_4(X_1,X_2)\succeq (\frac{5}{12}+\ve) (X_1^4+X_2^4)$.
\end{exa}

Example \ref{exa22} leads us to speculate the following.

\begin{conj}\label{c1}
Let $n,d\ge 2$. For all tuples of hermitian operators $X_1,\dots,X_n$ on a Hilbert space,
$$\ker \left(H_{2d}(X_1,\dots,X_n)-\mu_{n,d}(X_1^{2d}+\cdots+X_n^{2d})\right) = \ker X_1\cap\cdots\cap \ker X_n.$$
\end{conj}

There are two immediate consequences of Conjecture \ref{c1}. First, the equality $H_{2d}(X_1,\dots,X_n)=\mu_{n,d}(X_1^{2d}+\cdots+X_n^{2d})$ would hold if and only if $X_1=\cdots=X_n=0$.
Second, for every tuple of hermitian matrices $X_1,\dots,X_n$ there would exist $\ve>0$ such that $H_{2d}(X_1,\dots,X_n)\succeq (\mu_{n,d}+\ve)(X_1^{2d}+\cdots+X_d^{2n})$, by the same argument as in Example \ref{exa22}. 

Let us conclude with another conjecture. Table \ref{Table:SOS} displays SOHS representations of $H_{2d}(x_1,x_2)$ for $d\le 4$. A reader may notice that only polynomials with nonnegative coefficients appear in Table \ref{Table:SOS}. This is a consequence of $G_{n,d}$ being a \emph{completely positive} matrix, which has been verified for small $n,d$. Here, a matrix is completely positive if it factors as $S^*S$ for a nonnegative matrix $S$. This notion arises in combinatorial and nonconvex quadratic optimization \cite{cpmtx}. 

\begin{conj}\label{c2}
The matrix $G_{n,d}$ is completely positive for all $n,d\in\N$.
\end{conj}

In the context of this paper, Conjecture \ref{c2} implies that every NCHS polynomial admits an SOHS representation involving only polynomials with nonnegative coefficients, that is, an SOHS representation without any term cancellations.

\begin{table}
\begin{equation*}
    \begin{array}{c|c|l}
        d & \lambda_i & s_i \\
        \hline
        1   & 1  &  x_1 +\frac{1}{2} x_2 \\[3pt]
            & \frac{3}{4} & x_2\\
        \hline
        2   &  1 & x_1^2 +\frac{1}{4} x_1 x_2 +\frac{1}{4} x_2 x_1 +\frac{1}{6} x_2^2\\[3pt]
            & \frac{5}{9} & x_2^2 \\[3pt]
            & \frac{5}{48} & x_1 x_2 +x_2 x_1 +2 x_2^2  \\
        \hline
        3   & 1 & x_1^3 +\frac{1}{6} x_1^2 x_2 +\frac{1}{6} x_1 x_2 x_1 +\frac{1}{15} x_1 x_2^2 +\frac{1}{6} x_2 x_1^2 +\frac{1}{15} x_2 x_1 x_2 +\frac{1}{15} x_2^2 x_1 +\frac{1}{20} x_2^3 \\[3pt]
            & \frac{7}{16} & x_2^3 \\[3pt]
            & \frac{7}{180} & x_1^2 x_2 +x_1 x_2 x_1 +x_1 x_2^2 +x_2 x_1^2 +x_2 x_1 x_2 +x_2^2 x_1 +\frac{3}{2} x_2^3 \\[3pt]
            & \frac{7}{300} & x_1 x_2^2 +x_2 x_1 x_2 +x_2^2 x_1 +\frac{9}{2} x_2^3 \\
            \hline
        4   &   1 & x_1^4 +\frac{1}{8} x_1^3 x_2 +\frac{1}{8} x_1^2 x_2 x_1 +\frac{1}{28} x_1^2 x_2^2 +\frac{1}{8} x_1 x_2 x_1^2 +\frac{1}{28} x_1 x_2 x_1 x_2 +\frac{1}{28} x_1 x_2^2 x_1 \\[3pt]
        &&\qquad +\frac{1}{56} x_1 x_2^3 +\frac{1}{8} x_2 x_1^3 +\frac{1}{28} x_2 x_1^2 x_2 +\frac{1}{28} x_2 x_1 x_2 x_1 +\frac{1}{56} x_2 x_1 x_2^2 +\frac{1}{28} x_2^2 x_1^2 \\[3pt]
        &&\qquad+\frac{1}{56} x_2^2 x_1 x_2 +\frac{1}{56} x_2^3 x_1 +\frac{1}{70} x_2^4 \\[3pt] 
&   \frac{9}{25} & x_2^4 \\[3pt]            
&   \frac{9}{448} & x_1^3 x_2 +x_1^2 x_2 x_1 +\frac{2}{3} x_1^2 x_2^2 +x_1 x_2 x_1^2 +\frac{2}{3} x_1 x_2 x_1 x_2 +\frac{2}{3} x_1 x_2^2 x_1 +\frac{3}{5} x_1 x_2^3 \\[3pt]
&&\qquad +x_2 x_1^3 +\frac{2}{3} x_2 x_1^2 x_2 +\frac{2}{3} x_2 x_1 x_2 x_1 +\frac{3}{5} x_2 x_1 x_2^2 +\frac{2}{3} x_2^2 x_1^2 +\frac{3}{5} x_2^2 x_1 x_2 \\[3pt]
&&\qquad +\frac{3}{5} x_2^3 x_1 +\frac{4}{5} x_2^4 \\[3pt] 
&   \frac{3}{400} & x_1 x_2^3 +x_2 x_1 x_2^2 +x_2^2 x_1 x_2 +x_2^3 x_1 +8 x_2^4 \\[3pt] 
&   \frac{1}{245} & x_1^2 x_2^2 +x_1 x_2 x_1 x_2 +x_1 x_2^2 x_1 +\frac{9}{4} x_1 x_2^3 +x_2 x_1^2 x_2 +x_2 x_1 x_2 x_1 +\frac{9}{4} x_2 x_1 x_2^2 \\[3pt]
&&\qquad+x_2^2 x_1^2  +\frac{9}{4} x_2^2 x_1 x_2 +\frac{9}{4} x_2^3 x_1 +6 x_2^4 \\[3pt] 
    \end{array}
\end{equation*}
\caption{SOHS representations of $H_{2d}(x_1,x_2) = \sum_i \lambda_i s_i^*s_i$ for
$d=1,2,3,4$.}
\label{Table:SOS}
\end{table}

\bibliographystyle{alpha}
\bibliography{bibfile}
\end{document}